\documentclass{amsart}
\usepackage[all]{xy}
\usepackage{verbatim}
\usepackage{color}
\usepackage{amsthm}
\usepackage{amssymb}
\usepackage[colorlinks=true]{hyperref}
\usepackage{footmisc}
\usepackage{stmaryrd}

\numberwithin{equation}{section}

\newtheorem{theorem}[equation]{Theorem}
\newtheorem*{theorem*}{Theorem} \newtheorem{lemma}[equation]{Lemma}

\newtheorem*{conjecture*}{Mamma Conjecture}
\newtheorem*{conjecture1*}{Mamma Conjecture (revisited)}
\newtheorem{proposition}[equation]{Proposition}
\newtheorem{corollary}[equation]{Corollary}
\newtheorem*{corollary*}{Corollary}

\theoremstyle{remark}

\newtheorem{example}[equation]{Example}

\theoremstyle{remark}
\newtheorem{remark}[equation]{Remark}

\newcommand{\cA}{{\mathcal A}}
\newcommand{\cB}{{\mathcal B}}

\newcommand{\cF}{{\mathcal F}}
\newcommand{\cG}{{\mathcal G}}
\newcommand{\cH}{{\mathcal H}}

\newcommand{\cL}{{\mathcal L}}

\newcommand{\cO}{{\mathcal O}}

\newcommand{\cQ}{{\mathcal Q}}

\newcommand{\cW}{{\mathcal W}}
\newcommand{\cX}{{\mathcal X}}
\newcommand{\cY}{{\mathcal Y}}
\newcommand{\cZ}{{\mathcal Z}}

\newcommand{\bbA}{\mathbb{A}}
\newcommand{\bbB}{\mathbb{B}}
\newcommand{\bbC}{\mathbb{C}}

\newcommand{\bbF}{\mathbb{F}}

\newcommand{\bbP}{\mathbb{P}}

\newcommand{\bbQ}{\mathbb{Q}}
\newcommand{\bbZ}{\mathbb{Z}}

\newcommand{\dgcat}{\mathrm{dgcat}} 

\newcommand{\perf}{\mathrm{perf}}

\newcommand{\dg}{\mathrm{dg}}

\newcommand{\Hom}{\mathrm{Hom}}

\newcommand{\too}{\longrightarrow}

\newcommand{\ie}{\textsl{i.e.}\ }

\let\oldmarginpar\marginpar
\def\marginpar#1{\oldmarginpar{\tiny #1}}

\begin{document}

\title[HPD-invariance of the Tate conjecture]{HPD-invariance of the Tate conjecture}
\author{Gon{\c c}alo~Tabuada}
\address{Gon{\c c}alo Tabuada, Department of Mathematics, MIT, Cambridge, MA 02139, USA}
\email{tabuada@math.mit.edu}
\urladdr{http://math.mit.edu/~tabuada}
\thanks{The author was partially supported by a NSF CAREER Award}

%
\date{\today}
%
\abstract{We prove that the Tate conjecture is invariant under Homological Projective Duality (=HPD). As an application, we prove the Tate conjecture in the new cases of linear sections of determinantal varieties, and also in the cases of complete intersections of two quadrics. Furthermore, we extend the Tate conjecture from schemes to stacks and prove it for certain global orbifolds.}}

\maketitle
\vskip-\baselineskip
\vskip-\baselineskip


\section{Introduction and statement of results}
Let $k=\bbF_q$ be a finite field of characteristic $p$, $X$ a smooth projective $k$-scheme, and $l\neq p$ a prime number. Given this data, consider the associated cycle class map
\begin{equation}\label{eq:class-map}
\cZ^\ast(X)_{\bbQ_l} \too H^{2\ast}_{l\text{-}\mathrm{adic}}(X_{\overline{k}}, \bbQ_l(\ast))^{\mathrm{Gal}(\overline{k}/k)}
\end{equation}
with values in $l$-adic cohomology theory. Motivated by his work on the Shafarevich group of abelian varieties, Tate \cite{Tate} conjectured in the sixties the following:
\vspace{0.2cm}

Conjecture $T^l(X)$: The cycle class map \eqref{eq:class-map} is surjective.

\vspace{0.2cm}

The conjecture $T^l(X)$ holds when $\mathrm{dim}(X)\leq 1$ and also for $K3$-surfaces; see \cite{Tate-motives,Totaro}. Besides these cases, it remains wide open; consult Theorem \ref{thm:last} below for a proof of the Tate conjecture in several new cases.

A {\em differential graded (=dg) category} $\cA$ is a category enriched over dg $k$-vector spaces; see Keller's ICM address \cite{ICM-Keller}. Following Kontsevich \cite{Miami,finMot,IAS}, $\cA$ is called {\em smooth} if it is perfect as a bimodule over itself and {\em proper} if $\sum_i \mathrm{dim}_k H^i \cA(x,y)< \infty$ for any two objects $x, y \in \cA$. A classical example of a smooth proper dg category is the unique dg enhancement $\perf_\dg(X)$ of the category of perfect complexes $\perf(X)$ of a smooth proper $k$-scheme $X$ (or stack $\cX$); see Lunts-Orlov \cite[Thm.~2.12]{LO}. 

Given a smooth proper dg category $\cA$, a prime number $l\neq p$, and an~integer~$n\geq~1$, consider the abelian group $H^l_n(\cA):=\Hom\left(\bbZ(l^\infty), \pi_{-1} L_{KU}K(\cA\otimes_k k_n)\right)$, where $\bbZ(l^\infty)$ stands for the Pr\"ufer $l$-group, $k_n$ for the field $\bbF_{q^n}$, $K(\cA\otimes_k k_n)$ for the algebraic $K$-theory spectrum of the dg category $\cA\otimes_k k_n$, and $L_{KU}K(\cA\otimes_k k_n)$ for the Bousfield localization of $K(\cA\otimes_k k_n)$ with respect to topological complex $K$-theory $KU$. Note that the abelian group $H^l_n(\cA)$ can be alternatively defined as the $l$-adic Tate module of $\pi_{-1} L_{KU}K(\cA\otimes_k k_n)$. Under the above notations, the Tate conjecture admits the following noncommutative analogue:

\vspace{0.2cm}

Conjecture $T^l_{\mathrm{nc}}(\cA)$: The abelian groups $H^l_n(\cA)$, $n \geq 1$, are trivial.

\vspace{0.2cm}

Note that the conjecture $T^l_{\mathrm{nc}}(\cA)$ holds, for example, whenever the abelian groups $\pi_{-1} L_{KU}K(\cA\otimes_k k_n), n \ge 1$, are finitely generated. The motivation for the preceding noncommutative Tate conjecture arose from the following result:
\begin{theorem}[Thomason]\label{thm:Thomason}
Given a smooth projective $k$-scheme $X$, we have the following equivalence of conjectures $T^l(X) \Leftrightarrow T^l_{\mathrm{nc}}(\perf_\dg(X))$.
\end{theorem}  
\begin{proof}
Combine the canonical Morita equivalence $\perf_\dg(X)\otimes_k k_n\simeq \perf_\dg(X_{k^n})$ (see \cite[Lem.~4.26]{Gysin}) with the main theorem in Thomason's work \cite{Thomason}.
\end{proof}
Theorem \ref{thm:Thomason} shows that the Tate conjecture belongs not only to the realm of algebraic geometry but also to the broad setting of (smooth proper) dg categories. Making use of this latter noncommutative viewpoint, we now prove that the Tate conjecture is invariant under Homological Projective Duality (=HPD); for a survey on HPD we invite the reader to consult Kuznetsov \cite{Kuznetsov-ICM} and/or Thomas \cite{Thomas}.

Let $X$ be a smooth projective $k$-scheme equipped with a line bundle $\cL_X(1)$; we write $X \to \bbP(V)$ for the associated morphism where $V:=H^0(X,\cL_X(1))^\ast$. Assume that the triangulated category $\perf(X)$ admits a Lefschetz decomposition $\langle \bbA_0, \bbA_1(1), \ldots, \bbA_{i-1}(i-1)\rangle$ with respect to $\cL_X(1)$ in the sense of \cite[Def.~4.1]{Kuznetsov-IHES}. Following \cite[Def.~6.1]{Kuznetsov-IHES}, let $Y$ be the HP-dual of $X$, $\cL_Y(1)$ the HP-dual line bundle, and $Y\to \bbP(V^\ast)$ the morphism associated to $\cL_Y(1)$. Given a linear subspace $L\subset V^\ast$, consider the linear sections $X_L:=X\times_{\bbP(V)}\bbP(L^\perp)$ and $Y_L:=Y \times_{\bbP(V^\ast)}\bbP(L)$. 
\begin{theorem}[HPD-invariance]\label{thm:main}
Let $X$ and $Y$ be as above. Assume that $X_L$ and $Y_L$ are smooth\footnote{The linear section $X_L$ is smooth if and only if the linear section $Y_L$ is smooth; see \cite[page~9]{Kuznetsov-ICM}.}, that $\mathrm{dim}(X_L)=\mathrm{dim}(X)-\mathrm{dim}(L)$ and $\mathrm{dim}(Y_L)=\mathrm{dim}(Y)-\mathrm{dim}(L^\perp)$, and that the conjecture $T^l_{\mathrm{nc}}(\bbA_0^\dg)$ holds, where $\bbA^{\mathrm{dg}}_0$ stands for the dg enhancement of $\bbA_0$ induced from $\perf_\dg(X)$. Under these assumptions, we have the following equivalence of conjectures $T^l(X_L) \Leftrightarrow T^l(Y_L)$.
\end{theorem}
\begin{remark}\label{rk:singular}
\begin{itemize}
\item[(i)] Given a {\em general} subspace $L \subset V^\ast$, the sections $X_L$ and $Y_L$ are smooth, and $\mathrm{dim}(X_L)=\mathrm{dim}(X)-\mathrm{dim}(L)$ and $\mathrm{dim}(Y_L)=\mathrm{dim}(Y)-\mathrm{dim}(L^\perp)$. 
\item[(ii)] The conjecture $T^l_{\mathrm{nc}}(\bbA_0^{\mathrm{dg}})$ holds, in particular, whenever the triangulated category $\bbA_0$ admits a full exceptional collection.
\item[(iii)] Theorem \ref{thm:main} holds more generally when $Y$ is singular. In this case we need to replace $Y$ by a noncommutative resolution of singularities $\perf_\dg(Y;\cF)$, where $\cF$ stands for a certain sheaf of noncommutative algebras over $Y$ (see \cite[\S2.4]{Kuznetsov-ICM}), and conjecture $T^l(Y)$ by its noncommutative analogue $T^l_{\mathrm{nc}}(\perf_\dg(Y;\cF))$. 
\end{itemize}
\end{remark}
To the best of the author's knowledge, Theorem \ref{thm:main} is new in the literature. In what follows, we illustrate its strength in the case of two important HP-dualities.
\subsection*{Determinantal duality}
Let $U_1$ and $U_2$ be two $k$-vector spaces of dimensions $d_1$ and $d_2$, respectively, with $d_1\leq d_2$, $V:=U_1 \otimes U_2$, and $0 < r< d_1$ an integer. 

Consider the determinantal variety $\cZ^r_{d_1, d_2}\subset \bbP(V)$ defined as the locus of those matrices $U_2 \to U_1^\ast$ with rank $\leq r$. Recall that the determinantal varieties with $r=1$ are the classical Segre varieties. For example, $\cZ^1_{2,2}\subset \bbP^3$ is the quadric surface defined as the zero locus of the $2\times 2$ minor $v_0v_3 - v_1 v_2$. In contrast with the Segre varieties, the determinantal varieties $\cZ^r_{d_1, d_2}$, with $r\geq 2$, are not smooth. The singular locus of $\cZ^r_{d_1, d_2}$ consists of those matrices $U_2 \to U_1^\ast$ with rank $<r$, \ie it agrees with the closed subvariety $\cZ^{r-1}_{d_1, d_2}$. Nevertheless, it is well-known that $\cZ^r_{d_1, d_2}$ admits a canonical Springer resolution of singularities $\cX^r_{d_1, d_2} \to \cZ^r_{d_1, d_2}$, which comes equipped with a projection $q\colon \cX^r_{d_1, d_2} \to \mathrm{Gr}(r, U_1)$ to the Grassmannian of $r$-dimensional subspaces in $U_1$. Following \cite[\S3.3]{Marcello}, the category $\perf(X)$, with $X:=\cX^r_{d_1, d_2}$, admits a Lefschetz decomposition $\langle \bbA_0, \bbA_1(1), \ldots, \bbA_{d_2 r - 1}(d_2 r -1)\rangle$, where $\bbA_0=\bbA_1= \cdots = \bbA_{d_2 r -1}= q^\ast(\perf(\mathrm{Gr}(r,U_1)))\simeq \perf(\mathrm{Gr}(r,U_1))$. 
\begin{proposition}\label{prop:key}
The conjecture $T^l_{\mathrm{nc}}(\bbA_0^{\mathrm{dg}})$ holds.
\end{proposition}
Dually, consider the variety $\cW^r_{d_1, d_2}\subset \bbP(V^\ast)$, defined as the locus of those matrices $U^\ast_2 \to U_1$ with corank $\geq r$, and the associated Springer resolutions of singularities $Y:=\cY^r_{d_1, d_2} \to \cW^r_{d_1, d_2}$. As proved in Bernardara-Bolognesi-Faenzi in\footnote{In \cite[Prop.~3.4 and Thm.~3.5]{Marcello} the authors worked over an algebraically closed field of characteristic zero. However, the same proof holds {\em mutatis mutandis} over $k=\bbF_q$. Simply replace the reference \cite{Kapranov} to Kapranov's full strong exceptional collection on $\perf(\mathrm{Gr}(r, U_1))$ by the reference \cite[Thm.~1.3]{VdB} to Buchweitz-Leuschke-Van den Bergh's tilting bundle on $\perf(\mathrm{Gr}(r, U_1))$. The author is grateful to Marcello Bernardara for discussions concerning this issue.} \cite[Prop.~3.4 and Thm.~3.5]{Marcello}, $X$ and $Y$ are HP-dual to each other. Given a general linear subspace $L \subseteq V^\ast$, consider the smooth linear sections $X_L$ and $Y_L$. Note that whenever $\bbP(L^\perp)$ does not intersects the singular locus of $\cZ^r_{d_1, d_2}$, \ie the closed subvariety $\cZ^{r-1}_{d_1, d_2}$, we have $X_L=\bbP(L^\perp) \cap \cZ^r_{d_1, d_2}$.\begin{corollary}\label{cor:last}
We have the equivalence $T^l(X_L)\Leftrightarrow T^l(Y_L)$.
\end{corollary}
By construction, $\mathrm{dim}(X)= r(d_1+ d_2 -r)-1$ and $\mathrm{dim}(Y)= r(d_1-d_2-r) + d_1 d_2 -1$. Consequently, we have $\mathrm{dim}(X_L)= r(d_1+d_2-r) -1 - \mathrm{dim}(L)$ and $\mathrm{dim}(Y_L)= r(d_1 - d_2 - r) -1 +\mathrm{dim}(L)$. Since the Tate conjecture holds in dimensions $\leq 1$, we hence obtain the following result:
\begin{theorem}[Linear sections of determinantal varieties]\label{thm:last}
Let $X_L$ and $Y_L$ be smooth linear sections of determinantal varieties as in Corollary \ref{cor:last}.
\begin{itemize}
\item[(i)] Whenever $r(d_1 + d_2 -r) -1 - \mathrm{dim}(L)\leq 1$, the conjecture $T^l(Y_L)$ holds.
\item[(ii)] Whenever $r(d_1 -d_2 -r) -1 + \mathrm{dim}(L)\leq 1$, the conjecture $T^l(X_L)$ holds.
\end{itemize}
\end{theorem}
To the best of the author's knowledge, Theorem \ref{thm:last} is new in the literature. It proves the Tate conjecture in several new cases. Here are two families of examples: 
\begin{example}[Segre varieties]\label{ex:Segre}
Let $r=1$. Thanks to Theorem \ref{thm:last}(ii), whenever $d_1 - d_2 -2 +\mathrm{dim}(L) \leq 1$, the conjecture $T^l(X_L)$ holds. In these cases, $X_L$ is a linear section of the Segre variety $\cZ^1_{d_1, d_2}$ and the dimension of $X_L$ is $2(d_2 - \mathrm{dim}(L))$ or $2(d_2 - \mathrm{dim}(L))+1$. Therefore, for example, by letting $d_2 \to \infty$ (and by keeping $\mathrm{dim}(L)$ fixed), we obtain infinitely many new examples of smooth projective $k$-schemes $X_L$, of arbitrary high dimension, satisfying the Tate conjecture.
\end{example}
\begin{example}
Let $r=1$, $d_1=4$, and $d_2=2$. In this case, the Segre variety $\cZ^1_{4, 2} \subset \bbP^7$ agrees with the rational normal $4$-fold scroll $S_{1,1,1,1}$; see \cite[Ex.~8.27]{Harris}. Choose a general linear subspace $L \subset V^\ast$ of dimension $1$ such that the associated hyperplane $\bbP(L^\perp)\subset \bbP^7$ does not contains any $3$-plane of the rulling of $S_{1,1,1,1}$. By combining Example \ref{ex:Segre} with \cite[Prop.~2.5]{Faenzi}, we hence conclude that the rational normal $3$-fold scroll $X_L=S_{1,1,2}$  satisfies the Tate conjecture.
\end{example}
\begin{example}[Square matrices]\label{ex:square}
Let $d_1=d_2=d$. Thanks to Theorem \ref{thm:last}(ii), whenever $-r^2-1 + \mathrm{dim}(L)\leq 1$, the conjecture $T^l(X_L)$ holds. In these cases the dimension of $X_L$ is $2(dr - \mathrm{dim}(L))$ or $2(dr - \mathrm{dim}(L))+1$. Therefore, for example, by letting $d\to \infty$ (and by keeping $r$ and $\mathrm{dim}(L)$ fixed), we obtain infinitely many new examples of smooth projective $k$-schemes $X_L$, of arbitrary high dimension, satisfying the Tate conjecture.
\end{example}
\begin{example}
Let $d_1=d_2=3$ and $r=2$. In this case, the determinantal variety $\cZ^2_{3,3}\subset \bbP^8$ has dimension $7$ and its singular locus is the $4$-dimensional Segre variety $\cZ^1_{3,3} \subset \cZ^2_{3,3}$. Given a general linear subspace $L\subset V^\ast$ of dimension $5$, the associated smooth linear section $X_L$ is $2$-dimensional and, thanks to Example \ref{ex:square}, it satisfies the Tate conjecture. Note that since $\mathrm{codim}(L^\perp)=5>4 = \mathrm{dim}(\cZ^1_{3,3})$, the subspace $\bbP(L^\perp)\subset \bbP^8$ does not intersects the singular locus $\cZ^1_{3,3}$ of $\cZ^2_{3,3}$. Therefore, in all these cases, the surface $X_L$ is a linear section of the determinantal variety $\cZ^2_{3,3}$.
\end{example}
\subsection*{Veronese-Clifford duality}
Let $W$ be a $k$-vector space of dimension $d$ and $X$ the associated projective space $\bbP(W)$ equipped with the double Veronese embedding $\bbP(W) \to \bbP(S^2W), [w] \mapsto [w\otimes w]$. Consider the Beilinson's full exceptional collection $\perf(X)=\langle \cO_X(-1), \cO_X, \cO_X(1), \ldots, \cO_X(d-2)\rangle$ (see \cite{Beilinson}) and set $i:=\lceil d/2 \rceil$ and 
\begin{eqnarray*}
\bbA_0=\bbA_1=\cdots = \bbA_{i-2}:=\langle \cO_X(-1), \cO_X\rangle & \bbA_{i-1}:=\begin{cases}
\langle \cO_X(-1), \cO_X\rangle & \text{if}\,\,d=2i\\
\langle \cO_X(-1)\rangle & \text{if}\,\,d=2i-1
\end{cases} &
\end{eqnarray*}
Under these notations, the category $\perf(X)$ admits the Lefschetz decomposition $\langle \bbA_0, \bbA_1(1), \ldots, \bbA_{i-1}(i-1)\rangle$ with respect to the line bundle $\cL_X(1)=\cO_X(2)$. Note that this implies that the conjecture $T^l(\bbA_0^{\mathrm{dg}})$ holds; see Remark \ref{rk:singular}(ii). 

Let $\cH:=X\times_{\bbP(S^2W)} \cQ \subset X \times \bbP(S^2 W^\ast)$ be the universal hyperplane section, with $\cQ \subset \bbP(S^2 W) \times \bbP(S^2W^\ast)$ the incidence quadric. By construction, the projection $q\colon \cH \to \bbP(S^2 W^\ast)$ is a flat quadric fibration. As proved in \cite[Thm.~5.4]{Kuznetsov-quadrics} (see also \cite[Thm.~2.3.6]{Bernardara}) the HP-dual $Y$ of $X$ is given by $\perf_\dg(\bbP(S^2 W^\ast); \mathcal{C}l_0(q))$ (see Remark \ref{rk:singular}(iii)), where $\mathcal{C}l_0(q)$ stands for the sheaf of even Clifford algebras associated to $q$. 

Let $L \subset S^2W^\ast$ be a general linear subspace. On the one hand, $X_L$ corresponds to the smooth complete intersection of the $\mathrm{dim}(L)$ quadric hypersurfaces in $\bbP(W)$ parametrized by $L$. On the other hand, $Y_L$ is given by $\perf_\dg(\bbP(L); \mathcal{C}l_0(q)_{|L})$.
\begin{corollary}\label{cor:main}
We have the equivalence $T^l(X_L) \Leftrightarrow T^l_{\mathrm{nc}}(\perf_\dg(\bbP(L); \mathcal{C}l_0(q)_{|L}))$.
\end{corollary}
Recall that the space of quadrics $\bbP(S^2 W^\ast)$ comes equipped with a filtration $\Delta_d \subset \cdots \subset \Delta_2 \subset \Delta_1 \subset \bbP(S^2W^\ast)$, where $\Delta_i$ stands for the closed subscheme of those singular quadrics of corank $\geq i$. 
\begin{theorem}[Intersection of two quadrics]\label{thm:two}
Let $X_L$ be as in Corollary \ref{cor:main}. Assume that $\mathrm{dim}(L)=2$, that $\bbP(L) \cap \Delta_2 =\emptyset$, and that $p\neq 2$ when $d$ is odd. Under these assumptions, the conjecture~$T^l(X_L)$~holds.
\end{theorem}
The proof of Theorem \ref{thm:two} is based on the solution of the noncommutative Tate conjecture in Corollary \ref{cor:main}; see \S\ref{sec:proof}. An alternative (geometric) proof, based on the notion of variety of maximal planes, was obtained by Reid\footnote{Reid also assumed in {\em loc. cit.} that $\bbP(L) \cap \Delta_2 =\emptyset$; see \cite[Def.~1.9]{Reid}.} in the early seventies; see \cite[Thms.~3.14 and 4.14]{Reid}. Therein, Reid proved the Hodge conjecture but, as Kahn kindly informed me, a similar proof works for the Tate conjecture.
\begin{remark}[Intersection of even-dimensional quadrics]
In the case of an intersection $X_L$ of (several) even-dimensional quadrics, we prove in Theorem \ref{prop:2} below that the Tate conjecture $T^l(X_L)$, with $l\neq 2$, is equivalent to the corresponding Tate conjecture for the discriminant $2$-fold cover $\widetilde{\bbP}(L)$ of the projective space $\bbP(L)$. To the best of the author's knowledge, this (geometric) result is new in the literature.
\end{remark}
\subsection*{Tate conjecture for stacks}
Theorem \ref{thm:Thomason} allows us to easily extend Tate's original conjecture from smooth projective $k$-schemes $X$ to smooth proper algebraic $k$-stacks $\cX$ by setting $T^l(\cX):=T^l_{\mathrm{nc}}(\perf_\dg(\cX))$. The following result proves this extended conjecture for certain global orbifolds:
\begin{theorem}\label{thm:orbifold}
Let $G$ be a finite group of order $m$, $X$ a smooth projective $k$-scheme equipped with a $G$-action, and $\cX:=[X/G]$ the associated global orbifold. If $p \nmid m$, then we have the following implication of conjectures
\begin{eqnarray}\label{eq:implication-main}
\sum_{\sigma \subseteq G} T^l(X^\sigma \times \mathrm{Spec}(k[\sigma]))\Rightarrow T^l(\cX) && \forall\,\, l\nmid m\,,
\end{eqnarray}
where $\sigma$ is a cyclic subgroup of $G$. Moreover, whenever $\mathrm{dim}(X)\leq 3$ or $m\mid (q-1)$, the conjecture $T^l(X^\sigma\times \mathrm{Spec}(k[\sigma]))$ can be replaced by the conjecture $T^l(X^\sigma)$.
\end{theorem}
Note that if $m\mid (q-1)$, then $p\nmid m$. Note also that since the Tate conjecture holds in dimensions $\leq 1$, Theorem \ref{thm:orbifold} leads automatically to the following result:
\begin{corollary}\label{cor:examples} Assume that $p\nmid m$.
\begin{itemize}
\item[(i)] If $X=\bullet := \mathrm{Spec}(k)$, then the conjecture $T^l([\bullet/G])$ holds for every $l\nmid m$.
\item[(ii)] If $X$ is a curve $C$, then the conjecture $T^l([C/G])$ holds for every $l \nmid m$.
\item[(iii)] If $X$ is a surface $S$ such that the Tate conjecture $T^l(S)$, with $l\nmid m$, holds, then the conjecture $T^l([S/G])$, with $l\nmid m$, also holds.
\end{itemize}
\end{corollary}
\begin{example}
Let $S$ be an abelian variety equipped with the $\bbZ/2$-action $a \mapsto -a$. Thanks to Corollary \ref{cor:examples}(iii), the conjecture $T^l([S/(\bbZ/2)])$ holds for every $l\neq 2$.
\end{example}
%
We finish this section with the following ``twisted'' version of Corollary \ref{cor:examples}(iii):
\begin{theorem}\label{prop:twisted}
Let $G$ be a finite group of order $m$, $S$ a smooth projective $k$-surface equipped with a $G$-action, and $\cG$ a $G$-equivariant sheaf of Azumaya algebras over $S$ of rank $r$. Assume that $m\mid (q-1)$, that the $G$-action on $S$ is faithful, and that the Tate conjecture $T^l(S)$, with $l\nmid mr$, holds. Under these assumptions, the conjecture $T^l_{\mathrm{nc}}(\perf_\dg([S/G];\cG))$, with $l\nmid mr$, also holds.
\end{theorem}
\section{Variants of the noncommutative Tate conjecture}
Let $\cA$ be a smooth proper dg category $\cA$, $l\neq p$ a prime number, and $n \geq 1$ an integer. Given an integer $m \geq 1$, consider the following $\bbZ[1/m]$-module 
$$H^l_n(\cA; 1/m):=\Hom(\bbZ(l^\infty), (\pi_{-1} L_{KU}K(\cA\otimes_k k_n))_{1/m})$$ 
and the corresponding variant of the noncommutative Tate conjecture:

\vspace{0.2cm}

Conjecture $T^l_{\mathrm{nc}}(\cA; 1/m)$: The $\bbZ[1/m]$-modules $H^l_n(\cA; 1/m)$, $n \geq 1$, are trivial.

\vspace{0.2cm}

\begin{lemma}\label{lem:aux}
We have $T^l_{\mathrm{nc}}(\cA) \Leftrightarrow T^l_{\mathrm{nc}}(\cA;1/m)$ for every $l \nmid m$.
\end{lemma}
\begin{proof}
Since by assumption $l \nmid m$, the localization homomorphisms
\begin{eqnarray*}
\pi_{-1} L_{KU} K(\cA\otimes_k k_n) \too (\pi_{-1} L_{KU} K(\cA\otimes_k k_n))_{1/m} && n \geq 1
\end{eqnarray*}
induce an isomorphism between all the $l$-power torsion subgroups. Consequently, by passing to the $l$-adic Tate modules, we conclude that the induced homomorphisms $H^l_n(\cA) \to H^l_n(\cA; 1/m), n \geq 1$, are invertible.
\end{proof}
\section{Proof of Theorem \ref{thm:main}}
By definition of the Lefschetz decomposition $\langle \bbA_0, \bbA_1(1), \ldots, \bbA_{i-1}(i-1)\rangle$, we have a chain of admissible triangulated subcategories $\bbA_{i-1}\subseteq \cdots \subseteq \bbA_1\subseteq \bbA_0$ with $\bbA_r(r):=\bbA_r \otimes \cL_X(r)$. Note that $\bbA_r(r)\simeq \bbA_r$. Let $\mathfrak{a}_r$ be the right orthogonal complement to $\bbA_{r+1}$ in $\bbA_r$; these are called the {\em primitive subcategories} in \cite[\S4]{Kuznetsov-IHES}. By construction, we have the following semi-orthogonal decompositions:
\begin{eqnarray}\label{eq:decomp1}
\bbA_r = \langle \mathfrak{a}_r, \mathfrak{a}_{r+1}, \ldots, \mathfrak{a}_{i-1} \rangle && 0\leq r \leq i-1\,.
\end{eqnarray}
As proved in \cite[Thm.~6.3]{Kuznetsov-IHES} (see also \cite[Thm.~2.3.4]{Bernardara}), the category $\perf(Y)$ admits a HP-dual Lefschetz decomposition $\langle \bbB_{j-1}(1-j), \bbB_{j-2}(2-j), \ldots, \bbB_0\rangle$ with respect to $\cL_Y(1)$; as above, we have a chain of subcategories $\bbB_{j-1} \subseteq \bbB_{j-2} \subseteq \cdots \subseteq \bbB_0$. Moreover, the primitive subcategories coincide (via a Fourier-Mukai type functor) with those of $\perf(X)$ and we have the following semi-orthogonal decompositions:
\begin{eqnarray}\label{eq:decomp2}
\bbB_r=\langle \mathfrak{a}_0, \mathfrak{a}_1, \ldots, \mathfrak{a}_{\mathrm{dim}(V)-r-2}\rangle && 0 \leq r \leq j-1\,.
\end{eqnarray}
Furthermore, the assumptions $\mathrm{dim}(X_L)=\mathrm{dim}(X)-\mathrm{dim}(L)$ and $\mathrm{dim}(Y_L)=\mathrm{dim}(Y) - \mathrm{dim}(L^\perp)$ imply the existence of semi-orthogonal decompositions
\begin{equation}\label{eq:semi-1}
\perf(X_L)=\langle \bbC_L, \bbA_{\mathrm{dim}(V)}(1), \ldots, \bbA_{i-1}(i-\mathrm{dim}(V))\rangle 
\end{equation}
\begin{equation}\label{eq:semi-2}
\perf(Y_L)=\langle \bbB_{j-1}(\mathrm{dim}(L^\perp)-j), \ldots, \bbB_{\mathrm{dim}(L^\perp)}(-1), \bbC_L \rangle\,,
\end{equation}
where $\bbC_L$ is a common (triangulated) category. Let us denote, respectively, by $\bbC_L^\dg$, $\bbA_r^\dg$ and $\mathfrak{a}_r^\dg$ the dg enhancement of $\bbC_L$, $\bbA_r$ and $\mathfrak{a}_r$ induced from $\perf_\dg(X_L)$. Similarly, let us denote by $\bbC_L^{\dg'}$ and $\bbB_r^{\dg}$ the dg enhancement of $\bbC_L$ and $\bbB_r$ induced from $\perf_\dg(Y_L)$. Since by assumption the $k$-schemes $X_l$ and $Y_L$ are smooth (and projective), all these dg categories are smooth (and proper). 

Now, consider the following functors ($n \geq 1$):
\begin{eqnarray}\label{eq:additive}
E_n\colon \dgcat(k) \too \mathrm{Ab} && \cA \mapsto \pi_{-1} L_{KU} K(\cA\otimes_k k_n)\,,
\end{eqnarray}
defined on the category of (small) dg categories and with values in the category of abelian groups. Thanks to Proposition \ref{prop:additive} below, the functors \eqref{eq:additive} are additive invariants of dg categories. As explained in \cite[Prop.~2.2]{book}, this implies, in particular, that the above semi-orthogonal decompositions \eqref{eq:semi-1}-\eqref{eq:semi-2} give rise to the following direct sums decompositions of abelian groups ($n \geq1$):
\begin{equation}\label{eq:group1} 
E_n(\perf_\dg(X_L)) \simeq E_n(\bbC^\dg_L) \oplus E_n(\bbA^\dg_{\mathrm{dim}(V)}) \oplus \cdots \oplus E_n(\bbA^\dg_{i-1})\end{equation}
\begin{equation}\label{eq:group2}
E_n(\perf_\dg(Y_L)) \simeq E_n(\bbB_{j-1}^\dg) \oplus \cdots \oplus E_n(\bbB^\dg_{\mathrm{dim}(L^\perp)}) \oplus E_n(\bbC_L^{\dg'})\,.
\end{equation}
Consequently, by applying the functor $\mathrm{Hom}(\bbZ(l^\infty), -)$ to the isomorphisms \eqref{eq:group1}-\eqref{eq:group2}, we obtain the following equivalences of conjectures:
\begin{equation}\label{eq:conjecture1}
T^l_{\mathrm{nc}}(\perf_\dg(X_L))\Leftrightarrow T^l_{\mathrm{nc}}(\bbC_L^\dg)+ T^l_{\mathrm{nc}}(\bbA^\dg_{\mathrm{dim}(V)}) + \cdots + T^l_{\mathrm{nc}}(\bbA^\dg_{i-1})
\end{equation}
\begin{equation}\label{eq:conjecture2}
T^l_{\mathrm{nc}}(\perf_\dg(Y_L)\Leftrightarrow T^l_{\mathrm{nc}}(\bbB^\dg_{j-1}) + \cdots + T^l_{\mathrm{nc}}(\bbB^\dg_{\mathrm{dim}(L^\perp)}) + T^l_{\mathrm{nc}}(\bbC^{\dg'}_L)\,.
\end{equation}
On the one hand, since by assumption the conjecture $T^l_{\mathrm{nc}}(\bbA_0^\dg)$ holds, we conclude from the above semi-orthogonal decompositions \eqref{eq:decomp1}-\eqref{eq:decomp2} that the conjectures $T^l_{\mathrm{nc}}(\bbA^\dg_r)$ and $T^l_{\mathrm{nc}}(\bbB^\dg_r)$ hold for every $r$. This implies that the right-hand side of \eqref{eq:conjecture1}, resp. \eqref{eq:conjecture2}, reduces to the conjecture $T^l_{\mathrm{nc}}(\bbC^\dg_L)$, resp. $T^l_{\mathrm{nc}}(\bbC_L^{\dg'})$. On the other hand, since the composed functor $\perf(X_L) \to \bbC_L \to \perf(Y_L)$ is of Fourier-Mukai type, the dg categories $\bbC_L^\dg$ and $\bbC_L^{\dg'}$ are Morita equivalent. Using the fact that the functors \eqref{eq:additive} invert Morita equivalences, this implies that $T^l_{\mathrm{nc}}(\bbC^\dg_L)\Leftrightarrow T^l_{\mathrm{nc}}(\bbC_L^{\dg'})$. Finally, since $X_L$ and $Y_L$ are smooth projective $k$-schemes, the proof follows now from the equivalences $T^l(X_L) \Leftrightarrow T^l_{\mathrm{nc}}(\perf_\dg(X_L))$ and $T^l(Y_L) \Leftrightarrow T^l_{\mathrm{nc}}(\perf_\dg(Y_L))$ of Theorem \ref{thm:main}.
\begin{proposition}\label{prop:additive}
The above functors \eqref{eq:additive} are additive invariants of dg categories in the sense of \cite[Def.~2.1]{book}.
\end{proposition}
\begin{proof}
Let $F\colon \cA \to \cB$ be a Morita equivalence; see \cite[Def.~1.3.6]{book}. As proved in \cite[Prop.~7.1]{Artin}, the induced dg functor $F\otimes_k k_n \colon \cA\otimes_k k_n \to \cB \otimes_k k_n$ is also a Morita equivalence. Therefore, since algebraic $K$-theory inverts Morita equivalences (see \cite[\S2.2.1]{book}), the induced homomorphism $K(F\otimes_k k_n) \colon K(\cA\otimes_k k_n) \to K(\cB \otimes_k k_n)$ is invertible. By definition of the above functors \eqref{eq:additive}, we hence conclude that the induced group homomorphism $E_n(\cA) \to E_n(\cB)$ is also invertible.

Now, let $\cA$ and $\cB$ be two (small) dg categories, and $\mathrm{B}$ a dg $\cA\text{-}\cB$-bimodule. Consider the dg category $T(\cA,\cB;\mathrm{B})$ whose set of objects is $\mathrm{obj}(\cA)\amalg \mathrm{obj}(\cB)$, whose dg $k$-modules of morphisms defined as follows 
$$ T(\cA, \cB; \mathrm{B})(x,y):=\begin{cases}
\cA(x,y) & \text{if}\,\,\, x, y \in \cA \\
\cB(x,y) & \text{if}\,\,\, x, y \in \cB \\
\mathrm{B}(x,y) & \text{if}\,\,\, x \in \cA\,\,\text{and}\,\, y \in \cB \\
0 & \text{if}\,\,\, x \in \cB\,\,\text{and}\,\, y \in \cA \,,\\
\end{cases}
$$
and whose composition law is induced by the composition law of $\cA$ and $\cB$ and by the dg $\cA\text{-}\cB$-bimodule structure of $\mathrm{B}$. Note that, by construction, we have canonical dg functors $\iota_\cA \colon \cA \to T(\cA, \cB;\mathrm{B})$ and $\iota_\cB\colon \cB \to T(\cA, \cB; \mathrm{B})$. Under these notations, we need to show that the dg functors $\iota_\cA$ and $\iota_B$ induce an isomorphism
\begin{equation}\label{eq:iso-induced}
E_n(\cA) \oplus E_n(\cB) \too E_n(T(\cA,\cB;\mathrm{B}))\,.
\end{equation}
Consider the dg categories $\cA\otimes_k k_n$ and $\cB \otimes_k k_n$ and the dg $(\cA\otimes_k k_n)\text{-}(\cB\otimes_k k_n)$-bimodule $\mathrm{B}\otimes_k k_n$. Since algebraic $K$-theory is an additive invariant of dg categories, the dg functors $\iota_{\cA\otimes_k k_n}$ and $\iota_{\cB \otimes_k k_n}$ induce an isomorphism
\begin{equation}\label{eq:isom-last}
K(\cA\otimes_k k_n)\oplus K(\cB \otimes_k k_n) \stackrel{\simeq}{\too} K(T(\cA\otimes_k k_n, \cB\otimes_k k_n; \mathrm{B}\otimes_k k_n))\,.
\end{equation}
Therefore, by definition of the above functors \eqref{eq:additive}, we conclude from \eqref{eq:isom-last} that the homomorphism \eqref{eq:iso-induced} is also invertible. This concludes the proof.
\end{proof}
\section{Proof of Proposition \ref{prop:key}}
As proved in \cite[Thms.~1.3 and 1.7]{VdB}, the dg category $\perf_\dg(\mathrm{Gr}(r,U_1))$ is Morita equivalent to a finite dimensional $k$-algebra of finite global dimension $A$. Since $\bbA_0^\dg=\perf_\dg(\mathrm{Gr}(r,U_1))$, we hence obtain the following equivalences of conjectures:
\begin{equation}\label{eq:equivalence-1}
T^l_{\mathrm{nc}}(\bbA_0^\dg) = T^l_{\mathrm{nc}}(\perf_\dg(\mathrm{Gr}(r,U_1))) \Leftrightarrow T^l_{\mathrm{nc}}(A)\,.
\end{equation}
Recall that a finite field $k$ is, in particular, perfect. Therefore, using the fact that the above functors \eqref{eq:additive} are additive invariants of dg categories, we conclude from \cite[Thm.~3.15]{Azumaya} that $E_n(A)\simeq E(A/J(A))$, where $J(A)$ stands for the Jacobson radical of $A$. This implies that $T^l_{\mathrm{nc}}(A)\Leftrightarrow T^l_{\mathrm{nc}}(A/J(A))$. Now, let $V_1, \ldots, V_m$ be the simple (right) $A/J(A)$-modules and $D_1:=\mathrm{End}_{A/J(A)}(V_1), \ldots, D_m:=\mathrm{End}_{A/J(A)}(V_m)$ the associated division $k$-algebras. Thanks to the Artin-Wedderburn theorem, the quotient $A/J(A)$ is Morita equivalent to $D_1 \times \cdots \times D_m$. Moreover, the center $Z_i$ of $D_i$ is a finite field extension of $k$ and $D_i$ is a central simple $Z_i$-algebra. Since the Brauer group of a finite field is trivial, we hence conclude that $D_1 \times \cdots \times D_m$ is Morita equivalent to $Z_1 \times \cdots \times Z_m$. This implies the following equivalences:
\begin{equation}\label{eq:equivalence-2}
T^l_{\mathrm{nc}}(A) \Leftrightarrow T^l_{\mathrm{nc}}(A/J(A)) \Leftrightarrow T^l_{\mathrm{nc}}(D_1) + \cdots + T^l_{\mathrm{nc}}(D_m) \Leftrightarrow T^l(Z_1) + \cdots + T^l(Z_m)\,.
\end{equation}
The proof follows now from the combination of \eqref{eq:equivalence-1}-\eqref{eq:equivalence-2} with the fact that, since $\mathrm{dim}(Z_i)=0$, the Tate conjectures $T^l(Z_i), 1 \leq i \leq m$, hold.
\section{Proof of Theorem \ref{thm:two}}\label{sec:proof}
We assume first that $d$ is even. Following \cite[\S3.5]{Kuznetsov-quadrics} (see also \cite[\S1.6]{Bernardara}), let $\cZ$ be the center of $\mathcal{C}l_0(q)_{|L}$ and $\mathrm{Spec}(\cZ)=:\widetilde{\bbP}(L) \to \bbP(L)$ the {\em discriminant cover} of $\bbP(L)$. As explained in {\em loc. cit.}, $\widetilde{\bbP}(L) \to \bbP(L)$ is a $2$-fold cover which is ramified over the divisor $D:=\bbP(L) \cap \Delta_1$. Since by assumption $\mathrm{dim}(L)=2$, we have $\mathrm{dim}(D)=0$. Consequently, since $D$ is smooth, $\widetilde{\bbP}(L)$ is also smooth. Let us write $\cG$ for the sheaf of noncommutative algebras $\mathcal{C}l_0(q)_{|L}$ considered as a sheaf of noncommutative algebras over $\widetilde{\bbP}(L)$. As proved in {\em loc. cit.}, since by assumption $\bbP(L)\cap \Delta_2=\emptyset$, $\cG$ is a sheaf of Azumaya algebras over $\widetilde{\bbP}(L)$ of rank $2^{(d/2)-1}$. Moreover, the category $\perf(\bbP(L); \mathcal{C}l_0(q)_{|L})$ is equivalent (via a Fourier-Mukai type functor) to $\perf(\widetilde{\bbP}(L); \cG)$. This leads to a Morita equivalence between the dg categories $\perf_\dg(\bbP(L); \mathcal{C}l_0(q)_{|L})$ and $\perf_\dg(\widetilde{\bbP}(L); \cG)$. Consequently, making use of Corollary \ref{cor:main}, we obtain the following equivalence of conjectures:
\begin{equation}\label{eq:equivalence-new}
T^l(X_L)\Leftrightarrow T^l_{\mathrm{nc}}(\perf_\dg(\widetilde{\bbP}(L); \cG))\,.
\end{equation}
Since by assumption $\mathrm{dim}(L)=2$, the $2$-fold cover $\widetilde{\bbP}(L)$ is a smooth projective curve. Using the fact that the Brauer group of every smooth curve over a finite field is trivial (see \cite[page 109]{Milne}), we hence conclude that the right-hand side of \eqref{eq:equivalence-new} is equivalent to $T^l_{\mathrm{nc}}(\perf_\dg(\widetilde{\bbP}(L)))\Leftrightarrow T^l(\widetilde{\bbP}(L))$. The proof follows now from the fact that the Tate conjecture holds for smooth projective curves. 

We now assume that $d$ is odd and that $p\neq 2$. Following \cite[\S3.6]{Kuznetsov-quadrics} (see also \cite[\S1.7]{Bernardara}), let $\widehat{\bbP}(L)$ be the {\em discriminant stack} associated to the pull-back $q_{|L}$ along $\bbP(L) \subset \bbP(S^2W^\ast)$ of the flat quadric fibration $q\colon \cH \to \bbP(S^2W^\ast)$. As explained in {\em loc. cit.}, since by assumption $1/2 \in k$, $\widehat{\bbP}(L)$ is a smooth Deligne-Mumford stack. Moreover, using the fact that, by construction, $\widehat{\bbP}(L)$ is a square root stack and that the critical locus of the flat quadric fibration $q_{|L}$ is the divisor $D$, we conclude from \cite[Thm.~1.6]{Ueda} that the category $\perf(\widehat{\bbP}(L))$ admits a semi-orthogonal decomposition $\langle \perf(D), \perf(\bbP(L))\rangle$. Consequently, an argument similar to the one in the proof of Theorem \ref{thm:main} yields the following equivalences of conjectures:
\begin{equation}\label{eq:aux}
T^l(\widehat{\bbP}(L)):=T^l_{\mathrm{nc}}(\perf_\dg(\widehat{\bbP}(L))) \Leftrightarrow T^l(D) + T^l(\bbP(L)) \Leftrightarrow T^l(D)\,.
\end{equation} 
Let us write $\cG$ for the sheaf of noncommutative algebras $\mathcal{C}l_0(q)_{|L}$ considered as a sheaf of noncommutative algebras over $\widehat{\bbP}(L)$. As proved in \cite[\S3.6]{Kuznetsov-quadrics} (see also \cite[\S1.7]{Bernardara}), since by assumption $\bbP(L) \cap \Delta_2=\emptyset$, $\cG$ is a sheaf of Azumaya algebras over $\widehat{\bbP}(L)$. Moreover, the category $\perf(\bbP(L); \mathcal{C}l_0(q)_{|L})$ is equivalent (via a Fourier-Mukai type functor) to $\perf(\widehat{\bbP}(L); \cG)$. This leads to a Morita equivalence between the dg categories $\perf_\dg(\bbP(L); \mathcal{C}l_0(q)_{|L})$ and $\perf_\dg(\widehat{\bbP}(L); \cG)$. Making use of Corollary \ref{cor:main}, we hence obtain the following equivalence of conjectures:
\begin{equation}\label{eq:equivalence-new1}
T^l(X_L) \Leftrightarrow T^l_{\mathrm{nc}}(\perf_\dg(\widehat{\bbP}(L); \cG))\,.
\end{equation}
Since by assumption $\mathrm{dim}(L)=2$ and the Brauer group of every smooth curve over a finite field is trivial, the right-hand side of \eqref{eq:equivalence-new1} is equivalent to $T^l(\widehat{\bbP}(L))$. Consequently, since $\mathrm{dim}(D)=0$, the proof follows now from the combination of \eqref{eq:aux} with the fact that the Tate conjecture holds for $0$-dimensional $k$-schemes.
\subsection*{Intersection of even-dimensional quadrics}
\begin{theorem}\label{prop:2}
Let $X_L$ be as in Corollary \ref{cor:main}. Assume that $\bbP(L) \cap \Delta_2 =\emptyset$, that the divisor $\bbP(L) \cap \Delta_1$ is smooth, and that $d$ is even. Under these assumptions, we have the following equivalence of conjectures $T^l(X_L) \Leftrightarrow T^l(\widetilde{\bbP}(L))$ for every $l \neq 2$.
\end{theorem}
\begin{proof}
Similarly to the proof of Theorem \ref{thm:main}, we have the equivalence of conjectures
\begin{equation}\label{eq:equivalence-adapted}
T^l(X_L)\Leftrightarrow T^l_{\mathrm{nc}}(\perf_\dg(\widetilde{\bbP}(L); \cG))\,,
\end{equation}
where $\cG$ is a certain sheaf of Azumaya algebras over $\widetilde{\bbP}(L)$ of rank $2^{(d/2)-1}$. Consider the following functors ($n\geq 1$) with values in the category of $\bbZ[1/2]$-modules:
\begin{eqnarray}\label{eq:functors-2}
& E_n(-)_{1/2}\colon \dgcat(k) \too \mathrm{Mod}(\bbZ[1/2]) & \cA \mapsto (\pi_{-1}L_{KU} K(\cA\otimes_k k_n))_{1/2}\,.
\end{eqnarray}
As in the proof of Proposition \ref{prop:additive}, the functors \eqref{eq:functors-2} are additive invariants of dg categories. Consequently, since the rank of the sheaf of Azumaya algebras $\cG$  is a power of $2$ and the category $\mathrm{Mod}(\bbZ[1/2])$ is $\bbZ[1/2]$-linear, \cite[Thm.~2.1]{Azumaya} implies that the induced morphisms $E_n(\perf_\dg(\widetilde{\bbP}(L)))_{1/2} \to E_n(\perf_\dg(\widetilde{\bbP}(L); \cG))_{1/2}$, $n \geq 1$, are invertible. This yields the following equivalences of conjectures:
\begin{equation*}
T^l_{\mathrm{nc}}(\perf_\dg(\widetilde{\bbP}(L); \cG); 1/2) \Leftrightarrow T^l_{\mathrm{nc}}(\perf_\dg(\widetilde{\bbP}(L)); 1/2)) \Leftrightarrow T^l(\widetilde{\bbP}(L); 1/2)\,.
\end{equation*}
Making use of Lemma \ref{lem:aux}, we hence obtain the following equivalence:
\begin{eqnarray}\label{eq:equivalence-last}
T^l_{\mathrm{nc}}(\perf_\dg(\widetilde{\bbP}(L); \cG)) \Leftrightarrow T^l(\widetilde{\bbP}(L)) && \forall\,\, l \neq 2\,.
\end{eqnarray}
The proof follows now from the combination of the equivalences \eqref{eq:equivalence-adapted} and \eqref{eq:equivalence-last}.
\end{proof}
\begin{remark}[Azumaya algebras]\label{rk:Azumaya}
Let $X$ be a smooth projective $k$-scheme and $\cG$ a sheaf of Azumaya algebras over $X$ of rank $r$. Similarly to the proof of the above equivalence \eqref{eq:equivalence-last}, we have $T^l_{\mathrm{nc}}(\perf_\dg(X;\cG))\Leftrightarrow T^l(X)$ for every $l \nmid r$. As illustrated in (the proof of) Theorem \ref{prop:twisted} below, such an equivalence does {\em not} holds more generally in the case of a sheaf of Azumaya algebras over a global orbifold. 
\end{remark}

\section{Proof of Theorem \ref{thm:orbifold}}
Consider the functors with values in the category of $\bbZ[1/m]$-modules ($n\geq 1$):
\begin{eqnarray}\label{eq:functors-3}
& E_n(-)_{1/m}\colon \dgcat(k) \too \mathrm{Mod}(\bbZ[1/m]) & \cA \mapsto (\pi_{-1}L_{KU} K(\cA\otimes_k k_n))_{1/m}\,.
\end{eqnarray}
Similarly to the proof of Proposition \ref{prop:additive}, these functors are additive invariants of dg categories. We start by proving the first claim. Note that $p\nmid m$ if and only if $1/m \in k$. Therefore, since the functors \eqref{eq:functors-3} are additive invariants of dg categories and the category $\mathrm{Mod}(\bbZ[1/m])$ is $\bbZ[1/m]$-linear, we conclude from \cite[Thm.~1.1 and Rk.~1.4]{Orbifold} that $E_n(\perf_\dg(\cX))_{1/m}$ is a direct summand of the $\bbZ[1/m]$-module $\bigoplus_{\sigma\subseteq G} E_n(\perf_\dg(X^\sigma \times \mathrm{Spec}(k[\sigma])))_{1/m}$. This leads to the implication: 
\begin{equation}\label{eq:implication-aux1}
\sum_{\sigma \subseteq G} T^l(X^\sigma \times \mathrm{Spec}(k[\sigma]); 1/m) \Rightarrow T^l(\cX; 1/m)\,.
\end{equation}
By combining \eqref{eq:implication-aux1} with Lemma \ref{lem:aux}, we hence obtain the implication \eqref{eq:implication-main}.

Let us now prove the second claim. We start with the following remark:
\begin{remark}[Divisors]\label{rk:last} 
Given a smooth projective $k$-scheme $X$, consider the Tate conjecture for divisors $T^{l,1}(X)$: {\em The cycle class map \eqref{eq:class-map} with $\ast=1$ is surjective}. As proved by Tate in \cite[Prop.~5.1]{Tate-motives}, when $\mathrm{dim}(X)\leq 3$, we have $T^l(X) \Leftrightarrow T^{l, 1}(X)$. 
\end{remark}
If by assumption $\mathrm{dim}(X)\leq 3$, the dimension of the smooth projective $k$-schemes $X^\sigma$ and $X^\sigma \times \mathrm{Spec}(k[\sigma])$ is also $\leq 3$. Therefore, by combining Remark \ref{rk:last} with the following equivalences of conjectures (see \cite[Thm.~5.2]{Tate-motives})
$$ T^{l,1}(X^\sigma \times \mathrm{Spec}(k[\sigma]))\Leftrightarrow T^{l,1}(X^\sigma) + T^{l,1}(\mathrm{Spec}(k[\sigma]))\Leftrightarrow T^{l,1}(X^\sigma)\,,$$
we conclude that $T^l(X^\sigma \times \mathrm{Spec}(k[\sigma])) \Leftrightarrow T^l(X^\sigma)$. 

Assume now that $m\mid (q-1)$. Note that $m\mid (q-1)$ if and only if $k$ contains the the $m^{\mathrm{th}}$ roots of unity. Therefore, since $1/m \in k$, since the functors \eqref{eq:functors-3} are additive invariants of dg categories, and since the category $\mathrm{Mod}(\bbZ[1/m])$ is $\bbZ[1/m]$-linear, we conclude from \cite[Cor.~1.6(i)]{Orbifold} that $E_n(\perf_\dg(\cX))_{1/m}$ is a direct summand of the $\bbZ[1/m]$-module $\bigoplus_{\sigma \subseteq G} E_n(\perf_\dg(X^\sigma))^{\oplus r_\sigma}_{1/m}$, where $r_\sigma \geq 1$ are certain integers. This leads to the following implication of conjectures: 
\begin{equation}\label{eq:implication-aux}
\sum_{\sigma \subseteq G} T^l(X^\sigma; 1/m) \Rightarrow T^l(\cX; 1/m)\,.
\end{equation}
By combining \eqref{eq:implication-aux} with Lemma \ref{lem:aux}, we hence obtain the searched implication of conjectures $\sum_{\sigma \subseteq G} T^l(X^\sigma) \Rightarrow T^l(\cX)$, $\forall\,\, l\nmid m$.
\section*{Proof of Theorem \ref{prop:twisted}}
Since by assumption $m\mid (q-1)$, $k$ contains the $m^{\mathrm{th}}$ roots of unity and $1/m \in k$. Therefore, using the fact that the above functors \eqref{eq:functors-3} (with $m$ replaced by $mr$) are additive invariants of dg categories and that the category $\mathrm{Mod}(\bbZ[1/mr])$ is $\bbZ[1/mr]$-linear, we conclude from \cite[Cor.~1.29(ii)]{Orbifold} that $E_n(\perf_\dg([S/G]; \cG))_{1/mr}$ is a direct summand of the $\bbZ[1/mr]$-module $\bigoplus_{\sigma \subseteq G} E_n(\perf_\dg(Y_\sigma))_{1/mr}$, $n \geq 1$, where $Y_\sigma$ is a certain $\sigma^\vee$-Galois cover of $S^\sigma$. This leads to the implication of conjectures: 
\begin{equation}\label{eq:implication-big}
\sum_{\sigma \subseteq G} T^l(Y_\sigma; 1/mr) \Rightarrow T_{\mathrm{nc}}^l(\perf_\dg([S/G];\cG); 1/mr)\,.
\end{equation}
By combining \eqref{eq:implication-big} with Lemma \ref{lem:aux}, we hence obtain the implication:
\begin{eqnarray}\label{eq:implication-big1}
 \sum_{\sigma \subseteq G} T^l(Y_\sigma) \Rightarrow T_{\mathrm{nc}}^l(\perf_\dg([S/G];\cG)) && \forall\,\, l \nmid mr\,.
\end{eqnarray}
Now, since by assumption the $G$-action is faithful, we have $\mathrm{dim}(Y_\sigma)=\mathrm{dim}(S^\sigma)\leq 1$ for every non-trivial cyclic subgroup $\sigma$ of $G$. Consequently, \eqref{eq:implication-big1} reduces to the searched implication of conjectures $T^l(S) \Rightarrow T_{\mathrm{nc}}^l(\perf_\dg([S/G];\cG))$, $\forall\,\, l \nmid mr$.
%

\medbreak\noindent\textbf{Acknowledgments:} The author is grateful to Bruno Kahn for his interest in these results/arguments and for the reference \cite{Reid}. The author also would like to thank the Hausdorff Research Institute for Mathematics for its hospitality.

\end{document}

\end{proof}